%% file: main.tex
\documentclass[11pt]{article}

\usepackage{multirow}
\usepackage{amssymb,amsmath,amsthm}
\usepackage{color}
\usepackage{graphicx}
  \DeclareGraphicsExtensions{.jpeg,.png,.jpg,.eps,.pdf}
\usepackage{tikz}
\usepackage{epstopdf}
\usepackage{multicol}
\usepackage{multirow}
\usepackage{hyperref}

\usepackage{todonotes}

\input{notation.tex}

\def\gJ{\nabla J}
\def\clb{\underline{c}}
\def\cub{\bar{c}}

{\bf}{\it}
{\bf}{\it}
\newtheorem{example}{Example}{\bf}{\rm}
\newtheorem{lemma}{Lemma}{\bf}{\it}
\newtheorem{theorem}{Theorem}{\bf}{\it}
{\bf}{\it}
{\bf}{\it}
{\bf}{\rm}

\begin{document}


\title{
Convergence of Batch Updating Methods with \\
Approximate Gradients and/or Noisy Measurements: \\
Theory and Computational Results
}

\author{Tadipatri Uday Kiran Reddy and M.\ Vidyasagar
\thanks{
TUKR is an undergraduate in Electrical Engineering at the
Indian Institute of Technology Hyderabad, Kandi, Telangana 502284, India;
Email: ee19btech11038@iith.ac.in;
MV is SERB National Science Chair and Distinguished Professor,
Indian Institute of Technology Hyderabad, Kandi, Telangana 502284, India;
Email: m.vidyasagar@iith.ac.in.
The research of MV is supported by the Science and Engineering Research Board
(SERB), India.
}
}

\date{\today}

\maketitle

\begin{abstract}

In this paper, we present a unified and general framework for analyzing
the batch updating approach to nonlinear, high-dimensional optimization.
The framework encompasses all the currently used batch updating approaches,
and is applicable to nonconvex as well as convex functions.
Moreover, the framework permits the use of noise-corrupted gradients,
as well as first-order approximations to the gradient (sometimes referred
to as ``gradient-free'' approaches).
By viewing the analysis of the iterations as a problem in the
convergence of stochastic processes, we are able to establish a very
general theorem, which includes most known convergence results for
zeroth-order and first-order methods.
The analysis of ``second-order'' or momentum-based methods is not a part
of this paper, and will be studied elsewhere.
However, numerical experiments indicate that momentum-based methods
can fail if the true gradient is replaced by its first-order approximation.
This requires further theoretical analysis.
\end{abstract}

\section{Introduction}\label{sec:Intro}

\subsection{Scope of the Paper}\label{ssec:11}

In this paper we consider the general problem of minimizing a 
possibly nonconvex $C^1$ function $J : \R^d \ap \R$,
using batch updating of the argument.
In our nomenclature, batch updating refers to selecting \textit{some}
of the $d$ components
of the current guess $\bth_t$ at each time $t$, and then updating these 
components along a randomly chosen ``search direction'' $\bphi_{t+1}$.
Usually the search direction $\bphi_{t+1}$ is an approximation
to the negative gradient $-\gJ(\bth_t)$.
When the dimension $d$ of the parameter $\bth$ is very large,
computing the complete gradient
$\gJ(\bth_t)$ and/or updating every component of $\bth_t$
at every iteration can be
extremely resource-intensive  in both storage and CPU time.
For this reason, batch updating
(under whatever name) has become the \textit{de facto}
approach for large-scale minimization.
This approach goes under a variety of names in the
optimization and deep learning research communities,
including Coordinate Gradient Descent, 
(Batch) Stochastic Gradient Descent, (Randomized) Block Coordinate Descent, etc.
Further, much of the theory is developed for convex optimization.
However, in problems of deep learning, the objective function
is definitely not convex.
Thus, ideally, any approach should also be applicable to nonconvex
optimization.
There are further considerations as well.
When methods such as back-propagation are used to compute
the gradient $\gJ(\bth)$, computing \textit{only some components}
of $\gJ(\bth)$ does not necessarily result in significant savings
in computation.
On the other hand, function evaluation is much simpler than gradient
evaluation.
Therefore approximating some components of the gradient $\gJ(\bth_t)$
using first-order differences is a feasible approach.
(This approach also has a variety of names, such as zeroth-order
or derivative-free.)
While doing so, it would be realistic to permit some measurement
errors, that is, random disparities between the desired outputs of
the computation, and the actual computation.

In this paper, we propose a very general framework that encompasses
all the known batch updating paradigms.
Then we state and prove a very general result on the convergence of
batch updating, provided two specific conditions are satisfied.
It is shown that all widely used batch updating protocols do indeed
satisfy these conditions; therefore most available convergence results
for batch updating are included as special cases of our general result.
Moreover, the framework provides a readymade method of proving the
convergence of batch updating for protocols to be developed in future.

The analysis given here is applicable to methods that use true gradients
(possibly corrupted by noise) and approximate gradients (zeroth-order
or gradient-free methods).
Another popular approach to optimization is to use momentum-based methods,
wherein the current search direction is a linear combination of the
current gradient and the immediately preceding gradient (or noisy
versions thereof).
Methods such as Polyak's Heavy Ball method \cite{Polyak-cmmp64},\
Nesterov's accelerated method
\cite{Nesterov-Dokl83,Nesterov04a},
Sutskever's method \cite{Sutskever-et-al-PMLR13},
Bengio's method \cite{Bengio-CoRR12},
ADAM \cite{Diederik-ADAM-arxiv14}, NADAM \cite{Dozat-ICLR16},
and so on.
A recent paper \cite{Bara-Bian-SIAM21} establishes the convergence
of the ADAM algorithm using methods very similar to those used here.
However, that is for full updating and not batch updating.
Thus we plan to extend our approach to second-order or momentum-based
methods coupled with batch updating.

We then carry out a few numerical simulations to validate our analysis.
As expected, zeroth-order and first-order methods converge with
batch updating. Also, momentum-based methods with noise-corrupted gradients
and batch updating also converge. However, momentum-based methods that use
approximate gradients and batch updating fail to converge.
This shows that there is room for further theoretical analysis.

\section{Literature Review}\label{sec:Comp}

The literature on optimization is vast, and it is not possible
to review it in its entirety.
Given that fact, our review is focused on the papers that are directly
relevant to the specific class of problems studied here.
When $J(\cdot)$ has a minimum, a simple gradient descent algorithm applied to 
\textit{error-free} measurements of the gradient $\f(\bth) = - \gJ(\bth)$
will converge to $\bths$ under mild conditions 
\cite{Khamaru-Wainwright-icml18}.
Several momentum-based algorithms, which make use of not just the
current gradient and but also past gradients and/or function values,
have been proposed to improve the convergence rates.
Some examples are Polyak's Heavy Ball (HB)
\cite{Polyak-cmmp64}, Nesterov's Accelerated Gradient (NAG) 
\cite{Nesterov-Dokl83, Sutskever-et-al-PMLR13, Bengio-CoRR12}.
These  methods are
widely used in the optimization literature, especially in training 
Deep Neural-Networks.
The paper \cite{Ruder-arxiv16} gives a 
very good survey on variants of gradient descent algorithms for deep learning.

The phrase ``Stochastic Gradient Descent'' (SGD) refers to the case where
when the gradients are noisy.
The noises can arise due to 
(i) choosing partial data (known as \textit{mini-batches}),
(ii) uncertainty while  measuring gradients, and
(iii) approximation errors that arise
when the gradient is approximated using first-order differences as in
\eqref{eq:37}, often referred to as \textit{zeroth-order} methods,
(iv) compression of gradients for saving communication bandwidth in
CPU/GPUs.
With noisy gradients, one can use either full update (every component
is updated at each time), or batch update, which is the focus of this paper.
\newline
\textbf{Full update under noise:}
The literature includes
training DNNs with mini-batches of data \cite{Li-et-al-kdd14},
compressed gradients \cite{Lin-et-al-iclr18}, or sometimes intentional noise
injections during training for generalizing the model 
\cite{Noh-et-al-neurips17}.
The Robbins-Siegmund theorem \cite{Robb-Sieg71} 
plays a crucial role in proving the convergence.
The authors of 
\cite{Liu-Yuan-arxiv22, Sebbouh-Gower-Defazio-arxiv20} have proved the 
convergence using the Robbins-Siegmund theorem assuming when function
is Lipschitz-continuous,
Convergence rates for both convex and non-convex have been analyzed.
Moreover, their analysis was
extended to accelerated algorithms like HB and NAG.
However, the stochastic noise 
is assumed to have zero-mean and bounded variance. 
These assumptions do not hold when approximate gradients as in
\eqref{eq:37} are used.
Many other works such as 
\cite{Khaled-Rich-arxiv20, Bottou-et-al-review18, Mertikopoulos-et-al-arxiv20}
also assume that the errors in gradient measurements have bounded variance.
Relaxing the conditions on the noise has been less studied.
For example, In \cite{Wang-et-al-neurips21} has
provided converge guarantees for SGD under infinite variance noise, which arose
due to heavy tails.
However, this work has made strong assumptions on $J(\cdot)$
being $\C^2$ and strongly convex, which is not the case with
deep learning problems.
SGD for nonconvex functions has mostly been studied 
when error variance is bounded;
see for example \cite{Crawshaw-et-al-arxiv22, Fehrman-Gess-Jentzen-jmlr20, 
Li-Orabona-pmlr19}.
\newline
\textbf{Batch update under noise:}
Updating only a few randomly selected components at each time
is commonly referred to as Block Coordinate Descent (BCD);
see e.g. \cite{Tseng-jota01}.
Batch updating with error-free measurements 
has been studied widely with convergence results \cite{Nesterov-siamjo12, 
Richtrik-Takac-mp12, Richtrik-Takac-mp15}.
In \cite{Lu-Xiao-arxiv13}, the authors have 
provided a probabilistic convergence result, based
on the Nesterov's framework \cite{Nesterov-siamjo12}.
The study was limited to smooth convex functions and bounded noise variance.
Rather than choosing the blocks to be updated at random,
some researchers have studied methods of selecting the blocks
during the iterations.
Some authors \cite{Spall-jota12, Hernandez-Spall-acc14, Hernandez-Spall-acc16, 
Hernandez-ciss16} have proposed cyclic schemes for batch updates;
convergence has been studied for (pseudo-)convex functions and bounded 
variance of the errors.
Stochastic Approximation (SA) \cite{Robbins-Monro51, 
Wolf-AOMS52, Borkar-book09, Bottou-et-al-review18} is the key framework
in proving convergence.
Convergence of batch updating in SGD for non-convex functions is not studied
much. 
In \cite{Xu-Yin-siamjo15}, the convergence of BCD is proved for 
(non-)convex functions
under bounded noise in the measurements.
Batch updating has gained lot
of attention in  distributed ML \cite{Niu-et-al-arxiv11}, broadly categorized
into two main algorithms: Synchronous SGD (updates are performed
one after another node) \cite{Chen-et-al-iclr16} and Asynchronous SGD (ASGD) 
(random updates by any node at anytime) \cite{Niu-et-al-arxiv11, 
Xie-et-al-icml20}.
Hogwild! \cite{Niu-et-al-arxiv11} a variant of ASGD has 
been widely used in distributed training.
Convergence proofs for Hogwild!
are studied in \cite{Nguyen-et-al-pmlr18} for both (non-)convex and 
unbounded gradients.
\newline
\textbf{Approximate gradients}:
\be\label{eq:approx_grad}
\phi_{t+1,i} =
\frac{ J(\bth_t - c_t \eb_i)  -  J(\bth_t + c_t \eb_i)  } {2 c_t}
+ \frac{\xi^-_{t+1,i} - \xi^+_{t+1,i}  } {2 c_t} .
\ee
In batch updating, computing only a few components of the gradient
does not always yield savings in computation.
In contrast, if each component of the gradient is approximated as in
\eqref{eq:approx_grad}, then it is possible to compute these only for the
components to be updated.
However, in this case, the measurements have neither zero mean nor
bounded variance.
If we let $c_t \ap 0$ as $\tai$, the mean tends to $-\gJ(\bth_t)$, but
there is always an error.
Moreover, the variance of the noise term $(\xi_{t+1,i}^- - \xi_{t+1,i}^+)/c_t$
\textit{approaches infinity}.
This problem was studied initially in \cite{Kief-Wolf-AOMS52,Blum54}.
In \cite{Spall-jhu98}, the gradient is estimated using only $2$-function
evaluations,
which is further reduced to just one function evaluation in
\cite{Vakhitov-arc09, Spall-automatica97}.
However,
employing these methods in SGD suffers from very slow convergence rate.
Nesterov \& Spokoiny \cite{Nesterov-Spokoiny-fcm15} have deduced upper
bounds on various gradient estimates discussed above.
In addition, they
have proved convergence for \textit{full update} SGD and other momentum-
based methods.
In \cite{Wang-et-al-arxiv17}, convergence is proved when
using approximate gradients under stochastic noises, but
again with \textit{full updates}.
Moreover, their work
was limited to convex functions $\in \C^2$ with bounded hessian.
Other papers such as
\cite{Cai-et-al-siamjo22, Cai-arxiv22} had similar convergence results
under the same assumptions.
Convergence using approximate gradients with 
batch update was proved in \cite{Cai-et-al-arxiv17}, but only for
$\C^2$ convex functions.
The authors of \cite{Wang-et-al-neurips21} have proved convergence 
of SGD under unbounded noises (approximate gradients are a special 
case of this), but only for strong convex functions and with full updates.
\newline
\textbf{Momentum-Based methods}: In \cite{Mohammadi-et-al-acc19}, the 
performance of these algorithms is analyzed for the case where the 
gradients are corrupted by noise. It is shown that the allowed bound 
on the noise depends on the condition number $K$ of the convex function 
and that the error bound for momentum-based methods is larger than that 
for steepest descent by a factor of $\sqrt{K}$. Thus the presence of noise
in the gradient measurements largely negates the advantages of momentum-based
methods. Other approaches such as ADAM \cite{Diederik-ADAM-arxiv14} can 
also, cope with noisy measurements of the gradient. 
In \cite{Barakat-Bianchi-siamjo21}, the convergence of ADAM is proved
for (non-)convex functions under stochastic noise with full update.
Convergence of momentum-based methods under batch updates and unbounded
variance has been studied very rarely.
Numerical experiments in recent works
(including ours) have observed that NAG diverges.
For instance, in 
\cite{Assran-Rabbat20} the authors study the behavior of Nesterov's
method when there are errors in the gradient, and demonstrate that
sometimes the method can diverge.
There are other references as well, but these are an indicative sample.
\newline
\textbf{Similar works}: When approximate
gradients are used together with batch updating, the measurement error
is neither unbiased nor bounded in its variance.
So far as the authors are aware, there is no paper that treats this
situation, especially for nonconvex functions.
\ben
\item \cite{Liu-Yuan-arxiv22}, lacks stochastic unbounded variance and
	batch updates.
\item \cite{Nguyen-et-al-pmlr18}, lacks unified analysis and general
	theorem.
\item \cite{Nesterov-Spokoiny-fcm15}, not for batch update.
\item \cite{Wang-et-al-arxiv17, Cai-et-al-siamjo22, Cai-arxiv22}, only for full
	update assumes function belongs to convex and $\C^2$ class.
\item \cite{Cai-et-al-arxiv17}, assumes convexity and $\C^2$.
\item \cite{Xu-Yin-siamjo15}, applications to approximate was not studied.
\een

In addition to the above differences, a general theorem for SGD which works for
various options is not available in the literature. In this paper, we propose
a very general theorem SGD which encapsulates all variants, even momentum-based
methods but an analysis of this option is left for future study.

\section{General Convergence Theorems for Batch Updating}\label{thm:Thm}

In this section, we state
the main convergence theorem for batch updating.
The proof is given in the Appendix.

\subsection{Various Assumptions}\label{thm:21}

In this section, we state and interpret various assumptions about the
objective function and the search direction.
We begin with the objective function.

Suppose $J : \R^d \ap \R$ is a given $\C^1$ function, and it is
desired to find a stationary point of $J(\cdot)$, that is,
a solution of $\gJ(\bth) = \bz$.
It is \textit{not} assumed that the solution is unique.
Our objective is to ensure that the theory presented here embraces
nonconvex functions in addition to convex functions.
In presenting these assumptions, we make use of the following concept 
from \cite[Definition 1]{MV-MCSS-arxiv23}:
A function $\eta : \R_+ \ap \R_+$ is
said to \textbf{belong to Class $\B$} if $\eta(0) = 0$, and in addition
\bd
\inf_{\e \leq r \leq M} \eta(r) > 0 , \fa 0 < \e < M < \infty .
\ed

The following is a list of various assumptions about the objective function.
Note that not every assumption is used in every theorem.
\ben
\item[(J1).] $J(.)$ is $\C^1$ and $\gJ(\cdot)$ is globally
Lipschtiz-continuous with constant $2L$. 
\item[(J2).] $J(.)$ has a global minimum $J^*$, which is attained.
Moreover, $J(\cdot)$ has compact level sets, that is, for each constant
$c \in \R$, the set $\{ \bth : J(\bth) \leq c \}$ is compact.\footnote{Other
equivalent phrases are: $J(\cdot)$ is radially unbounded, or
$J(\cdot)$ is coercive.}
\item[(J3).] Define
\bd
\Jb(\bth) := J(\bth) - J^* .
\ed
Then there exists a constant $C$ such that
\bd
\nmeusq{\gJ(\bth)} \leq C_1 [ J(\bth) - J^* ] , \fa \bth \in \R^d .
\ed
\item[(J4).] There is a function $\mu(\cdot)$ belonging to Class $\B$ such that
\bd
\mu(\Jb(\bth)) \leq \nmeu{\gJ(\bth)} , \fa \bth \in \R^d .
\ed
\item[(J5).] There is a function $\nu(\cdot)$ belonging to Class $\B$
such that
\be\label{eq:212}
\r(\bth,S(J)) \leq \nu(\Jb(\bth) )  , \fa \bth \in \R^d .
\ee
where
\bd
\r(\bth,S(J)) := \inf_{\bpsi \in S(J)} \nmeu{\bth - \bpsi } 
\ed
is the distance between $\bth$ and the set $S(J)$.
\een

Now we interpret these assumptions.
In the scalar case $d = 1$, if there is a unique minimum $\bths$,
Assumption (J3) is equivalent to $J(\th) \leq (C_1/2) \th^2$.
In higher dimensions, the assumption requires $J(\bth)$ to grow no
faster than quadratically in $\r(\bth,S(J))$.
Assumptions (J4) and (J5) permit $J(\cdot)$ to have multiple stationary points;
thus $S(J)$ need not be a singleton set.
However, both imply that $J(\bth) = J^*$ for all $\bth \in S(J)$;
thus every local minimum must also be a global minimum.
Assumption (J4) implies that, for any sequence $\{ \bth_t \}$,
if $\Jb(\bth_t)$ is bounded away from zero, then so is $\nmeu{\gJ(\bth_t)}$,
while Assumption (J5) implies that if $\Jb(\bth_t) \ap 0$, then
$\r(\bth_t,S(J)) \ap 0$ as $\tai$.

\begin{example}\label{exam:21}
Consider the following function $J : \R \ap \R$:
\bd
J(\th) := \left\{ \ba{ll} \sin[(\pi/2)(\th-1)] + 1 , & -5 \leq \th \leq 5 , \\
0.5 + \sqrt{(\pi/2)(\th-5) + 0.25}, & 5 \leq \th , \\
J(-\th), & \th \leq -5 , \ea \right.
\ed
which function is depicted in Figure \ref{fig:exam-1}
It satisfies Assumptions (J4) and (J5), and has multiple global minima.
However, if some of the local minima were to be greater than the global
minimum, then (J4) and (J5) would fail to be satisfied.
\bfig[htb]
\bc
\includegraphics[width=90mm]{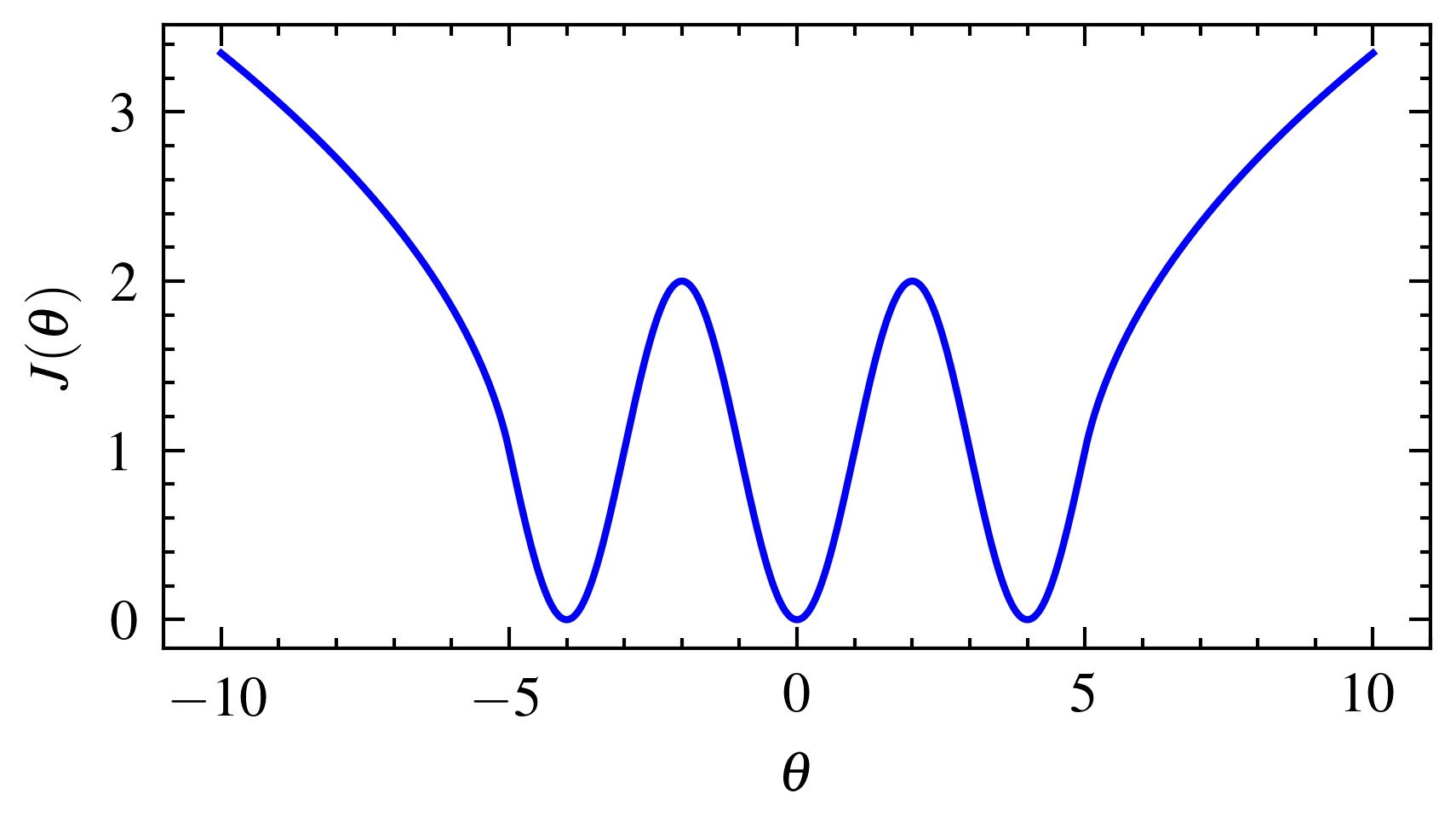}
\ec
\caption{Depiction of a function with multiple minima satisfying
Assumptions (J4) and (J5)}
\label{fig:exam-1}
\efig
\end{example}

\begin{example}\label{exam:22}
Now suppose $J: \R^d \ap \R^d$ is $\C^2$ is strictly convex,
and suppose further that there exist constants $0 < \clb \leq \cub < \infty$
such that
\bd
\clb I_d \leq \nabla^2 J(\bth) \leq \cub I_d , \fa \bth \in \R^d ,
\ed
Here, for symmetric matrices $A, B$, the notation
$A \leq B$ means that $B-A$ is positive semidefinite.
Since $J(\cdot)$ is strictly convex, $S(J)$ is a singleton $\{ \bths\}$.
In this case Assumptions (J4) and (J5) are satisfied trivially.
Thus the theory presented here is applicable to strictly convex functions,
as well as \textit{some} nonconvex functions.
\end{example}

The algorithm for finding a stationary point of $J(\cdot)$ is as follows:
Choose an initial guess $\bth_0 \in \R^d$ (usually deterministic).
At time $t+1 \geq 1$, we choose a \textit{random}
``search direction'' $\bphi_{t+1}$, which satisfies certain
assumptions to be specified below.
Once $\bphi_{t+1}$ is chosen, $\bth_t$ is updated according to
\be\label{eq:22}
\bth_{t+1} = \bth_t + \al_t \bphi_{t+1} .
\ee
Here, $\{ \al_t \}$ is a \textit{deterministic} sequence of step sizes
where $\al_t \in (0,1)$ for each $t$.
In batch updating we could, in principle, use ``local clocks'' to choose
a different step size for each updated component, as suggested in
\cite{Borkar98}.
However, in the interests of simplicity,
in this paper we stick to using a \textit{global clock},
so that every component that is updated uses the same step size.
Batch updating means that, at any time $t$, only some components of
$\bphi_{t+1}$ could be nonzero, and the rest are set to zero.
Some possible options for selecting $\bphi_{t+1}$ are discussed in
Section \ref{sec:Var}.

Let $\bth_0^t$ denote the tuple $\bth_0 , \cdots , \bth_t$,
and define $\bphi_1^t$ analogously.
(Note that there is no $\bphi_0$.)
Let $\{ \F_t \}$ be a filtration (that is, an increasing sequence of
$\s$-algebras) such that $\bth_0^t,\bphi_1^t$, and the random processes
up to time $t$ that are involved in selecting the nonzero components of
$\bphi_{t+1}$, are all measurable with respect to $\F_t$.
For future use, let $\M(\F_t)$ denote the set of all functions
that are measurable with respect to $\F_t$.

The general framework studied in \cite{Pol-Tsy73} (which is our inspiration)
is that the search direction $\bphi_t$ is
random, but there is a fixed constant $\d$ such that
\bd
E( \IP{ \gJ(\bth_t)} {\bphi_{t+1} } | \F_t ) \leq - \d \nmeusq{\gJ(\bth_t)} ,
\ed
where $E(\cdot | \F_t )$ denote the conditional expectation with
respect to $\F_t$.
Since the constant $\d$ can be absorbed into the step size $\al_t$,
we choose $\d = 1$, and accordingly define:
\be\label{eq:24}
\bphit_t := E( \bphi_{t+1} | \F_t ) + \gJ(\bth_t) .
\ee
Thus $\bphit_t$ is the difference between the conditional expectation
of the search direction, and the desired value $\gJ(\bth_t)$.
Observe that $\bphit_t \in \M(\F_t)$.
Next, define $\bzeta_{t+1}$ as the ``unpredictable part''
of the search direction, namely
\be\label{eq:25}
\bzeta_{t+1} := \bphi_{t+1} - E( \bphi_{t+1} | \F_t ) .
\ee

With this background, we make the following assumptions about the search
direction $\bphi_{t+1}$:
\ben
\item[(D1).] There is a sequence of deterministic constants $\{ b_t \}$ such
that 
\be\label{eq:27}
\nmeusq{ \bphit_t} \leq b_t^2 , \fa t .
\ee
\item[(D2).] There is a sequence of deterministic constants $\{ \s_t \}$ such
that
\be\label{eq:28}
E( \nmeusq{\bzeta_{t+1} } | \F_t ) \leq \s_t^2 (1 + \nmeusq{\gJ(\bth_t) } ) ,
\fa t .
\ee
\een

\subsection{Theorem Statement}\label{ssec:23}

With the above assumptions, we now state the convergence theorem.

\begin{theorem}\label{thm:1}
Define $\{ \bth_t \}$ via \eqref{eq:22}.
Then we can state the following conclusions.
\ben
\item
Suppose assumptions (J1) -- (J3) and (D1) -- (D2) are satisfied.
Suppose further that
\be\label{eq:29}
\sum_{t=0}^\infty \al_t^2 < \infty ,
\sum_{t=0}^\infty \al_t b_t < \infty ,
\sum_{t=0}^\infty \al_t^2\s_t^2 < \infty .
\ee
Then $\{ J(\bth_t) \}$, $\{ \gJ(\bth_t) \}$,
and $\{ \bth_t \}$ are  bounded almost surely.
\item Suppose that in addition to the assumptions in Item 1, we add
\be\label{eq:210}
\sum_{t=0}^\infty \al_t = \infty .
\ee
Then
\be\label{eq:210a}
\liminf_{\tai} \nmeu{\gJ(\bth_t)} = 0 \as 
\ee
\item Suppose that in addition to the assumptions in Item 2, we add Assumption
(J4).
Then $J(\bth_t) \ap J^*$ almost surely, and
$\gJ(\bth_t) \ap \bz$ almost surely, as $\tai$.
\item Suppose that in addition to the assumptions in Item 3, we add Assumption
(J5).
Then $\r(\bth_t,S(J)) \ap 0$ almost surely as $\tai$.
\een
\end{theorem}

\subsection{Proof of Main Convergence Theorem}\label{ssec:23}

The proof of Theorem \ref {thm:1}
is based on the ``almost supermartingale lemma'' of Robbins \&
Siegmund Theorem \cite{Robb-Sieg71}.
That paper is rather difficult to locate.
However, the same theorem is stated as Lemma 2 in \cite
[Section 5.2]{BMP92}.
A recent survey of many results along similar lines is found in
\cite{Fran-Gram22}, where Lemma \ref{lemma:1} below is stated as Lemma 4.1.
The result states the following:

\begin{lemma}\label{lemma:1}
Suppose $\{ z_t \} , \{ \d_t \} , \{ \g_t \} , \{ \psi_t \}$ are
stochastic processes defined on some probability space $(\OM,\SI,P)$,
taking values in $[0,\infty)$, adapted to some
filtration $\{ \F_t \}$, satisfying
\be\label{eq:27}
E(z_{t+1} | \F_t ) \leq (1 + \d_t) z_t + \g_t - \psi_t \as, \fa t .
\ee
Define
\be\label{eq:27a}
\OM_0 := \{ \om \in \OM : \sum_{t=0}^\infty \d_t(\om) < \infty \}
\cap \{ \om : \sum_{t=0}^\infty \g_t(\om) < \infty \} .
\ee
Then for all $\om \in \OM_0$, $\lim_{\tai} z_t(\om)$ exists, and in addition,
\be\label{eq:27b}
\sum_{t=0}^\infty \psi_t(\om) < \infty , \fa \om \in \OM_0 .
\ee
In particular, if $P(\OM_0) = 1$, then $\{ z_t \}$ is bounded almost surely.
\end{lemma}

We also make use of the following result which is easy to prove, and
is stated in this form in \cite[Eq.\ (2.40)]{Ber-Tsi-SIAM00}.

\begin{lemma}\label{lemma:2}
Suppose $J(.)$ is $\C^1$ and $\gJ(.)$ is globally Lipschtiz-continuos with
constant $2L$.
Then
\be\label{eq:216}
J(\y_t) \leq J(\x_t) + \IP{\gJ(\x_t)}{\y_t - \x_t}
+ L \nmeusq{\y_t - \x_t}
\ee
\end{lemma}

We begin with a few preliminary observations.
By the ``tower'' property of conditional expectations,
it follows from \eqref{eq:25} that
\be\label{eq:26}
E( \bzeta_{t+1} | \F_t ) = \bz .
\ee
Therefore
\beq
E( \nmeusq{\bphi_{t+1}} | \F_t ) & = & \nmeusq{ \bphit_t - \gJ(\bth_t) }
+ E( \nmeusq{\bzeta_{t+1}} | \F_t ) \nonumber \\
& = & \nmeusq{\bphit_t} - 2 \IP{\bphit_t}{\gJ(\bth_t)} \nonumber \\ 
& + & \nmeusq{\gJ(\bth_t)} + E( \nmeusq{\bzeta_{t+1}} | \F_t ) \label{eq:23}
\eeq
\newline
\textbf{(Proof of Theorem \ref{thm:1}:)}
\newline
{\textbf{Item 1:}
Recall the update rule \eqref{eq:22} and apply Lemma \ref{lemma:1}.
This gives
\begin{eqnarray*}
J(\bth_{t+1}) & = & J( \bth_t + \al_t \bphi_{t+1} ) \\
& \leq & J(\bth_t) + \al_t \IP{\gJ(\bth_t)}{\bphi_{t+1}}
+ \al_t^2L \nmeu{\bphi_{t+1}}^2.
\end{eqnarray*}
Subtract $J^*$ from both sides and define $\Jb(\bth) := J(\bth) - J^*$.
This gives
\bd
\Jb(\bth_{t+1}) \leq \Jb(\bth_t) + \al_t \IP{\gJ(\bth_t)}{\bphi_{t+1}}
+ \al_t^2 L \nmeusq{\bphi_{t+1}} .
\ed
Now take conditional expectations and invoke \eqref{eq:24}, \eqref{eq:26}
and \eqref{eq:23}.
This gives
\begin{eqnarray*}
E( \Jb(\bth_{t+1}) | \F_t ) & \leq & \Jb(\bth_t)
+ \al_t \IP{\gJ(\bth_t)}{ \bphit_t} - \al_t \nmeusq{\gJ(\bth_t)} \\
& + & \al_t^2 L [ \nmeusq{ \bphit_t - \gJ(\bth_t)} 
+ E( \nmeusq{\bzeta_{t+1}} | \F_t) ] .
\end{eqnarray*}
Now apply Schwarz's inquality, and assumptions (D1), (D2).
This gives
\begin{eqnarray*}
E( \Jb(\bth_{t+1}) | \F_t ) & \leq & \Jb(\bth_t) + \al_t b_t \nmeu{\gJ(\bth_t)}
- \al_t \nmeusq{\gJ(\bth_t)} \nonumber \\
& + & \al_t^2 L [ b_t^2 + 2 b_t \nmeu{\gJ(\bth_t)} 
+ \nmeusq{\gJ(\bth_t)} + \s_t^2 ( 1 + \nmeusq{\gJ(\bth_t)} ) ] .
\end{eqnarray*}
Next, invoke (J3), and observe that
\bd
\nmeu{\gJ(\bth_t)} \leq 2 \nmeu{\gJ(\bth_t)} \leq 1 + \nmeusq{\gJ(\bth_t)} 
\leq 1 + C_1 \Jb(\bth) .
\ed
This leads to
\begin{eqnarray*}
E( \Jb(\bth_{t+1}) | \F_t ) & \leq & \Jb(\bth_t)
+ \al_t b_t (1 + C_1 \Jb(\bth_t)) - \al_t \nmeusq{\gJ(\bth_t)} \nonumber \\
& + & \al_t^2 L [ b_t^2 + 
b_t (1 + C_1 \Jb(\bth_t)) + C_1 \Jb(\bth_t) + \s_t^2 ( 1 + C_1 \Jb(\bth_t)) ] .
\end{eqnarray*}
After regrouping terms, the above bound is of the form \eqref{eq:27} with
\bd
z_t = \Jb(\bth_t) ,
\d_t = \al_t b_t C_1 + \al_t^2 L C_1 ( b_t + 1 + \s_t^2 ) ,
\ed
\bd
\g_t = \al_t b_t + \al_t^2 L ( b_t^2 + b_t + \s_t^2 ) ,
\psi_t = \al_t \nmeusq{\gJ(\bth_t)} .
\ed
For Item 1, let us ignore $\psi_t$ because it is non-negative.
Also, both $\{ \d_t \}$ and $\{ \g_t \}$ are deterministic sequences.
Hence, if it can be shown that both sequences $\{ \d_t \}$ and $\{ \g_t \}$
are summable, then it would follow from Lemma \ref{lemma:1}
that $\Jb(\bth_t)$ converges almost surely to some random
variable $\eta$, and is bounded along almost all sample paths.
Now (J3) implies that $\gJ(\bth_t)$ is also bounded almost surely,
while the compactness of the level sets of $J(\cdot)$ implies that
$\{ \bth_t \}$ is bounded almost surely.
Thus we focus on establishing the summability of these two sequences
$\{ \d_t \}$ and $\{ \g_t \}$.

For this purpose, we recall the hypotheses in \eqref{eq:29}.
Using these, it is shown that each term in the definitions of $\d_t$
and $\g_t$ is summable.
First,
$\sum_{t=0}^\infty \al_t b_t < \infty$, which is the first term (modulo
the constant $C_1$) in both $\d_t$ and $\g_t$.
Next,
\bd
\sum_{t=0}^\infty \al_t b_t < \infty \imp
\sum_{t=0}^\infty \al_t^2 b_t^2 < \infty ,
\ed
which is the second term in $\g_t$.
Next, because $\al_t$ is bounded by one, we have
\bd
\sum_{t=0}^\infty \al_t b_t < \infty \imp
\sum_{t=0}^\infty \al_t^2 b_t < \infty .
\ed
which is the second term in $\d_t$ and the third term in $\g_t$.
Next, by assumption $\sum_{t=0}^\infty \al_t^2 < \infty$, which is the
third term in $\d_t$.
Last, by assumption $\sum_{t=0}^\infty \al_t^2 \s_t^2 < \infty$,
which is the fourth and final term (modulo $C_1$) in both $\d_t$ and $\g_t$.
Hence both sequences are summable, which establishes Item 1.
\newline
\textbf{Item 2:}
Now we make use of the second part of Lemma \ref{lemma:1}, namely
\bd
\sum_{t=0}^\infty \psi_t = 
\sum_{t=0}^\infty \al_t \nmeusq{\gJ(\bth_t)}  < \infty \as
\ed
Suppose there exists an $\om \in \OM$ such that
\bd
\liminf_{\tai} \nmeusq{\gJ(\bth_t)(\om)} > 0 , \mbox{ say } 2 \e .
\ed
Choose a time $T$ such that
\bd
\nmeusq{\gJ(\bth_t)(\om)} \geq \e , \fa t \geq T ,
\ed
and observe that
\bd
\sum_{t=T}^\infty \al_t = \infty .
\ed
Then
\be\label{eq:210b} 
\sum_{t=T}^\infty \al_t \nmeusq{\gJ(\bth_t)(\om)}  \geq 
\e \sum_{t=T}^\infty \al_t = \infty ,
\ee
which shows that the set of such $\om \in \OM$ must have measure zero.
This is Item 2.
\newline
\textbf{Item 3:}
Now it is assumed that Assumption (J4) also holds in addition to (J3).
Define
$\OM_0$ to be the set of $\om \in \OM$ such that
$\lim_{t \ap \infty} J(\bth_t) = \eta(\om) < \infty$, and observe that $P(\OM_0) = 1$.
Now suppose $\om \in \OM_0$, and define
$M = M(\om) = \sup_t J(\bth_t) < \infty$.
Suppose by way of contradiction that $\eta(\om) > 0$, say 
$\eta(\om) > 2 \e$.
Choose $T$ sufficiently large that $J(\bth_t) \geq \e$ for all $t \geq T$.
Then by (J4), it follows that, for all $t \geq T$, we have
\bd
\nmeu{\gJ(\bth_t)} \geq \min_{\e \leq r \leq M} \mu(r) =: c > 0.
\ed
Now apply \eqref{eq:210b} with $\e$ replaced by $c$,
which leads to a contradiction.
Hence the set of $\om \in \OM_0$ for which $\eta(\om) > 0$
has measure zero, or $J(\bth_t) \ap J^*$ almost surely, and
$\gJ(\bth_t) \ap \bz$ almost surely.
\newline
\textbf{Item 4:}
We already know from Item 3 that $J(\bth_t) \ap J^*$ almost surely.
Now apply Assumption (J5).

\section{Various Batch Updating Options}\label{sec:Var}

In this section, we first propose some possible choices of search directions
$\bphi_{t+1}$ for batch updating. 
These include many commonly used choices.
Then we analyze each of these choices, and
show that they all satisfy the standard assumptions (D1) and (D2).
Usually (D1) is straight-forward and the main challenge is to establish (D2).
Once it is established that both (D1) and (D2) are satisfied, then one
can apply Theorem \ref{thm:1} to conclude that the iterative algorithm
\eqref{eq:22} converges to a solution for any objective function $J(\cdot)$
that satisfies Assumptions (J1) through (J5), or some subset thereof,
provided the step size conditions \eqref{eq:29} and \eqref{eq:210} are
satisfied.

\subsection{Various Options for Batch Updating}\label{ssec:31}
In this subsection we study different options that are possible to
fit in our convergence theorem.

\textbf{Option 1:}
The first of these is the standard steepest descent with possibly noise-corrupted
gradient measurements.
Let
\be\label{eq:31}
\bphi_{t+1} = - \gJ(\bth_t) + \bxi_{t+1} ,
\ee
where $\{ \bxi_t \}_{t \geq 1}$ is a sequence of unstructured disturbances,
including measurement errors.
Following the usage in this area, we will refer to these as ``measurement
noise.'' The noise $\bxi_{t+1}$ is assumed to satisfy
\be\label{eq:321}
E( \bxi_{t+1} | \F_t ) = \bz \fa t ,
\ee
\be\label{eq:322}
E( \nmeusq{ \bxi_{t+1} } | \F_t ) \leq M_t^2 ( 1 + \nmeusq{ \gJ(\bth_t) }) 
\fa t ,
\ee
for a known set of constants $M_t$. These assumptions applies to all options
studied here.

\textbf{Option 2:}
The next option is ``coordinate gradient descent'' as defined in
\cite{Wright15} and studied further in \cite{Bach-et-al-aisats19}.
At time $t$, choose an index $\kappa_{t+1} \in [d]$ at random with a
uniform probability, and independently of previous choices.
Let $\eb_{\kappa_{t+1}}$ denote the elementary unit vector with a $1$
as the $\kappa_{t+1}$-th component and zeros elsewhere,
and let $\bxi_{t+1}$ denote the measurement noise.
Then define
\be\label{eq:32}
\bphi_{t+1} = d \eb_{\kappa_{t+1} } \circ [ - \gJ(\bth_t) + \bxi_{t+1} ] .
\ee
The factor of $d$ arises because the likelihood that $\kappa_{t+1}$
equalling any one index $i \in [d]$ is $1/d$.
Thus, with the introduction of the factor $d$, it follows that if
we ignore the noise term $\bxi_{t+1}$ for the time being, then
\bd
E(\bphi_{t+1} | \F_t) = - \gJ(\bth_t) .
\ed
Note that, though $\kappa_{t+1}$ is used at time $t+1$,
it is \textit{chosen} at time $t$, that is, \textit{prior}
to the update at time $t+1$.
Hence $\kappa_{t+1} \in \M(\F_t)$.

\textbf{Option 3:}
This is option is just coordinate gradient descent along multiple
coordinates chosen at random.
At time $t$, choose $N$ different indices $\kappa_{t+1,n}$ from $[d]$
\textit{with replacement}, with each choice being independent of the rest,
and also of past choices.
Moreover, each $\kappa_{t+1,n}$ is chosen from $[d]$ with uniform probability.
Then define
\be\label{eq:33}
\bphi_{t+1} := \frac{d}{N} \sum_{n=1}^N \eb_{\kappa_{t+1}^n}
\circ [ -\gJ(\bth_t) + \bxi_{t+1,n}] ,
\ee
where $\{ \xi_{t,n} , t \geq 1 , n \in [N] \}$ is the noise sequence.
In this option, $\bphi_{t+1}$ can have \textit{up to} $N$ nonzero
components.
Because the sampling is \textit{with replacement}, there might
be some duplicated samples.
In such a case, the corresponding component of $\gJ(\bth_t)$ simply
gets counted multiple times in \eqref{eq:33}.

\textbf{Option 4:}
This is a variant of coordinate gradient descent.
At time $t+1$, let $\{ B_{t+1,i}, i \in [d] \}$  be independent Bernoulli
processes with success rate $\r_{t+1}$.
Thus
\be\label{eq:34}
\Pr \{ B_{t+1,i} = 1 \} = \r_{t+1} , \fa i \in [d] .
\ee
It is permissible for the success probability $\r_{t+1}$ to vary with time.
However, at any one time, all components must have the same success
probability.
Define
\be\label{eq:35}
\v_{t+1} := \sum_{i=1}^d \eb_i I_{ \{ B_{t+1,i} = 1 \} }  \in \bi^d .
\ee
Thus $\v_{t+1,i}$ equals $1$ if $ B_{t+1,i} = 1$, and equals $0$ otherwise.
Now define
\be\label{eq:36}
\bphi_{t+1} = \frac{1}{\r_{t+1}}  \v_{t+1} \circ [ - \gJ(\bth_t) +
\bxi_{t+1} ] .
\ee
In this option,
the search direction $\bphi_{t+1}$ can have up to $d$ nonzero components.
However, the \textit{expected} number of nonzero components is $\r_{t+1} d$.

\textbf{Options 1A through 4A:}
In these options, we apply Options 1 through 4, after replacing
the actual gradient $\gJ(\bth_t)$, 
by a first-order approximation, as suggested first in
\cite{Kief-Wolf-AOMS52}.
Let $\{ c_t \}$ be a predetermined sequence of ``increments'' (not to be
confused with the step size sequence $\{ \al_t \}$).
At each time $t+1$, define the approximate gradient $\y_{t+1}\in \R^d$ by
\be\label{eq:37}
y_{t+1, i} = \frac{J(\bth_t - c_t \eb_i) + \xi_{t+1,i}^-
- J(\bth_t + c_t \eb_i) - \xi_{t+1,i}^+ } {2 c_t } ,
\ee
where $\bxi_{t+1}^- , \bxi_{t+1}^+$ are zero-mean measurement noises
satisfying
\eqref{eq:322} that
\be\label{eq:327}
E( \bxi_{t+1}^- | \F_t ) = \bz \fa t ,
E( \bxi_{t+1}^+ | \F_t ) = \bz \fa t ,
\ee
\begin{eqnarray}\label{eq:328}
\max \{ E( \nmeusq{ \bxi_{t+1}^- } | \F_t ) ,
E( \nmeusq{ \bxi_{t+1}^+ } | \F_t ) \} \nonumber \\
\leq M_t^2 ( 1 + \nmeusq{ \gJ(\bth_t) } ) \fa t ,
\end{eqnarray}
for suitable constants $M_t$.

\subsection{Convergence Theorems}

\begin{theorem}\label{thm:2}
Suppose the objective function $J(\cdot)$ satisfies assumptions (J1)--(J5),
and suppose $\bphi_{t+1}$ is chosen according to
any one of Options 1 through 4.
\ben
\item
Suppose the step size sequence $\{ \al_t \}$ satisfies
\be\label{eq:323a}
\sum_{t=0}^\infty \al_t^2 < \infty ,
\sum_{t=0}^\infty \al_t^2 M_t^2 < \infty .
\ee
Then $\{ J(\bth_t) \}$, $\{ \gJ(\bth_t) \}$ and $\{ \bth_t \}$ are
all bounded almost surely.
\item If, in addition to \eqref{eq:323a}, the step size sequence $\{ \al_t \}$ 
also satisfies
\be\label{eq:324a}
\sum_{t=0}^\infty \al_t = \infty ,
\ee
then $\gJ(\bth_t) \ap \bz$ and $\r(\bth_t,S(J)) \ap 0$ as $\tai$, almost
surely as $\tai$.
\een
\end{theorem}

Note that if $M_t \leq M$, a fixed constant for all $t$, then the two conditions
in \eqref{eq:323a} are equivalent.
Moreover, taken together, \eqref{eq:323a} and \eqref{eq:324a} become
the well-known Robbins-Monro conditions \cite{Robbins-Monro51}.

\begin{theorem}\label{thm:3}
Suppose the objective function $J(\cdot)$ satisfies assumptions (J1)--(J5),
and suppose we apply batch updating with any one of Options 1A through 4A.
Under these conditions,
\ben
\item
Suppose that the step size sequence $\{ \al_t \}$ satisfies
\be\label{eq:3212}
\sum_{t=0}^\infty \al_t^2 < \infty , \sum_{t=0}^\infty \al_t c_t < \infty ,
\sum_{t=0}^\infty ( \al_t M_t  /c_t)^2 < \infty .
\ee
Then $\{ J(\bth_t) \}$, $\{ \gJ(\bth_t) \}$ and $\{ \bth_t \}$ are
all bounded almost surely.
\item If, in addition to \eqref{eq:3212}, we have that
\be\label{eq:3213}
\sum_{t=0}^\infty \al_t = \infty ,
\ee
then $\gJ(\bth_t) \ap \bz$ and $\r(\bth_t,S(J)) \ap 0$ as $\tai$, almost
surely as $\tai$.
\een
\end{theorem}

Note that if $M_t \leq M$, a fixed constant for all $t$, then the conditions
in \eqref{eq:3212} become the so-called ``Blum conditions,'' first proposed
in \cite{Blum54}.

\subsection{Analysis of Various Batch Updating Options}\label{ssec:32}

In this subsection, we analyze the various batch updating options
in Section \ref{ssec:31}.
To keep the notation from getting overly cumbersome, it is assumed that in
Options 1 through 4, the noise $\bxi_{t+1}$ satisfies
\be\label{eq:321}
E( \bxi_{t+1} | \F_t ) = \bz \fa t ,
\ee
\be\label{eq:322}
E( \nmeusq{ \bxi_{t+1} } | \F_t ) \leq M_t^2 ( 1 + \nmeusq{ \gJ(\bth_t) } 
\fa t ,
\ee
for a known set of constants $M_t$.
With these assumptions, each of Options 1 through 4, ``on average''
the search direction $\bphi_{t+1}$ equals $- \gJ(\bth_t)$.
Therefore Assumption (D1) is satisfied with $b_t = 0$ for all $t$.
This is \textit{not true} in Options 1A through 4A.
The discussion below will reveal that more general models can be
accommodated at the expense of more cumbersome notation.

In \textbf{Option 1}, it is easy to see that
\bd
E( \bphi_{t+1} | \F_t ) = - \gJ(\bth_t) ,
\bphit_t = \bz , \fa t .
\ed
Therefore Assumption (D1) is satisfied with $b_t = 0$ and
$M_t^2$ as in \eqref{eq:322}.

In \textbf{Option 2}, there are two factors contributing to the error
$\bzeta_{t+1}$.
First, even if $\bxi_{t+1} = \bz$, the search direction $\bphi_{t+1}$ equals
$- \gJ(\bth_t)$ only ``on average,'' and there is a variance due to this.
Second, the variance due to the noise $\bxi_{t+1}$ also needs to be
taken into account.

Let $i$ denote $\kappa_{t+1}$.
Then
\bd
\phi_{t+1,j} = \left\{ \ba{ll}
- (d-1) ( \gJ(\bth_t) )_j , & \mbox{if } j = i , \\
( \gJ(\bth_t) )_j , & \mbox{if } j \neq i . \ea \right.
\ed
Therefore
\bd
\sum_{j=1}^d ( \phi_{t+1,j} - ( \gJ(\bth_t) )_j )^2 =
(d-1)^2 ( \gJ(\bth_t)_i )^2 + \sum_{j \neq i} ( \gJ(\bth_t)_j )^2 .
\ed
\bd
\nmeusq{ \bphi_{t+1} - \gJ(\bth_t) } = \sum_{j=1}^d 
[ (d-1)^2 ( \gJ(\bth_t)_i )^2 + \sum_{j \neq i} ( \gJ(\bth_t)_j )^2 ] .
\ed
Now $\kappa_{t+1} = i$ with probability $1/d$.
Therefore the conditional probability 
\begin{eqnarray*}
E( \nmeusq{ \bphi_{t+1} - \gJ(\bth_t) } | \F_t) & = &
\frac{1}{d} \sum_{i=1}^d \sum_{j=1}^d
[ (d-1)^2 ( \gJ(\bth_t)_i )^2 + \sum_{j \neq i} ( \gJ(\bth_t)_j )^2 ] \\
& = &  \nmeusq{ \gJ(\bth_t) } \left[ \frac{ (d-1)^2 + (d-1) }{d} \right]
\nonumber \\
& = & (d-1) \nmeusq{ \gJ(\bth_t) } .
\end{eqnarray*}
Since
\bd
E( \nmeusq{\bzeta_{t+1}} | \F_t )
= E( \nmeusq{ \bphi_{t+1} - \gJ(\bth_t) } | \F_t),
\ed
if $\bxi_{t+1} \neq \bz$, its conditional variance simply adds to the above.
Therefore (D2) is satisfied with
\bd
\s_t^2 = (d-1) + M_t^2 .
\ed

In \textbf{Option 3}, $\bphi_{t+1}$ is the average of $N$ different quantities
wherein the error terms $\bzeta_{t+1}^n , n \in [N]$ are independent.
Therefore their variances just add up, an again Assumption (D2)
holds with $\s_t^2 = (d-1) + M_t^2$.

Next we come to \textbf{Option 4}.
For notational simplicity, we just use $\r$ in the place of $\r_{t+1}$.
In this case, each component $\phi_{t+1,i}$ equals $- (1/\r) \gJ(\bth_t)_i$
with probability $\r$, and $0$ with probability $1-\r$.
Thus $\zeta_{t+1,i}$ equals $1-(1/\r) \gJ(\bth_t)_i$ with probability $\r$,
and $\gJ(\bth_t)_i$ with probability $1-\r$.
As can be easily verified, the variance is
$(1-\r)/\r$ for each component.
As the Bernoulli processes for each component are mutually independent,
the variances simply add up.
It follows that
\bd
E( \nmeusq{\bzeta_{t+1}} | \F_t ) = \frac{d(1-\r)}{\r} .
\ed
Hence Assumption (D2) holds with 
\bd
\s_t^2 = \frac{d(1-\r)}{\r} + M_t^2 .
\ed

Finally we come to \textbf{Option 1A through 4A}, which is distinguished by the fact that
we cannot take $b_t = 0$ in general.
Define $\z_{t+1}$ by
\be\label{eq:325}
z_{t,i} = \frac{J(\bth_t + c_t \eb_i) - J(\bth_t - c_t \eb_i)}{2 c_t } ,
\fa i \in [d] .
\ee
Then it is clear that the search direction $\y_{t+1}$ in \eqref{eq:37}
is given by
\be\label{eq:326}
\y_{t+1} = - \z_t + \frac{\bxi_t^- - \bxi_t^+ }{2c_t} .
\ee
Recall from \eqref{eq:327} and \eqref{eq:328} that
\be\label{eq:327a}
E( \bxi_{t+1}^- | \F_t ) = \bz \fa t ,
E( \bxi_{t+1}^+ | \F_t ) = \bz \fa t ,
\ee
\be\label{eq:328a}
\max \{ E( \nmeusq{ \bxi_{t+1}^- } | \F_t ) ,
E( \nmeusq{ \bxi_{t+1}^+ } | \F_t ) \}
\leq M_t^2 ( 1 + \nmeusq{ \gJ(\bth_t) } ) \fa t ,
\ee
for suitable constants $M_t$.

\begin{lemma}\label{lemma:3}
With $\y_{t+1}$ defined as in \eqref{eq:326}, define, in analogy with
\eqref{eq:24} and \eqref{eq:25},
\be\label{eq:327b}
\bphit_t = E( \nmeu{\y_{t+1}} | \F_t ) + \gJ(\bth_t) ,
\ee
\be\label{eq:329}
\bzeta_{t+1} = \y_{t+1} - E( \nmeu{\y_{t+1} } | \F_t ) .
\ee
Then
\be\label{eq:3210}
\nmeusq{\bphit_t} \leq \sqrt{d} L c_t ,
\ee
\be\label{eq:3211}
E( \nmeusq{\bzeta_{t+1} } | \F_t ) \leq \frac{M_t^2}{2c_t^2} 
\ee
\end{lemma}

\begin{proof}
By \eqref{eq:327}, it follows that
\bd
E(\y_{t+1} | \F_t ) = - \z_t , \bphit_t = - \z_t + \gJ(\bth_t) .
\ed
Next, we adapt the proof of Lemma \ref{lemma:2}
taken from \cite{Ber-Tsi-SIAM00}
to the problem at hand, and write
\begin{eqnarray*}
z_{t,i} & = & \frac{J(\bth_t + c_t \eb_i) - J(\bth_t - c_t \eb_i)}{2 c_t } \\
& = & \frac{1}{2 c_t} \int_{-1}^1 \frac{d}{d\l}
[ J(\bth_t + \l c_t \eb_i ) ] \; d \l \\
& = & \frac{1}{2 c_t} \int_{-1}^1 c_t \IP{\eb_i}
{\gJ(\bth_t + \l c_t \eb_i )} \; d \l \\
& = & \half \int_{-1}^1 [ \IP{\eb_i}{\gJ(\bth_t)} \nonumber \\
& + & \IP {\eb_i} { \gJ(\bth_t + \l c_t \eb_i ) - \gJ(\bth_t)} ] \; d \l .
\end{eqnarray*}
Now apply Assumption (J1) and Schwarz's inequality.
This gives
\begin{eqnarray*}
| \IP {\eb_i} { \gJ(\bth_t + \l c_t \eb_i ) - \gJ(\bth_t)} | \\
\leq \nmeu { \gJ(\bth_t + \l c_t \eb_i ) - \gJ(\bth_t)} 
\leq  2L \l c_t .
\end{eqnarray*}
Therefore
\bd
z_{t+1,i} \leq [ \gJ(\bth_t) ]_i + \half \int_{-1}^1 2L \l c_t \; d \l
= [ \gJ(\bth_t) ]_i + L c_t .
\ed
Therefore
\bd
\nmeu{\bphit_t} \leq \sqrt{d} L c_t .
\ed
This proves \eqref{eq:3210}.
To prove \eqref{eq:3211}, observe that $\bxi_{t+1}^-$ and $\bxi_{t+1}^+$
are independent.
Therefore the conditional variance of $\bzeta_{t+1}$ is the sum of the
conditional variances of these two terms, divided by $4 c_t^2$,
which is bounded by $M_t^2/2c_t^2$.
\end{proof}

\section{Computational Results}\label{sec:Comp}

To validate the theory, we have conducted the following numerical experiment:
We set
$J(\bth) = \bth^\top A \bth + \log\left(\sum_{i=0}^{1000}e^{\th_i}\right)$,
where $A$ is a $1000 \times 1000$ positive definite matrix with a condition
number of 100, whose eigenvectors are \textit{not} aligned closely with the
elementary basis vectors; that is, $A$ is far from being a diagonal matrix.
Four different updating options were implemented, namely:
Option 1 (full gradient updating), Option 4 (batch updating with $1000$
different independent Bernoulli  processes $\kappa_{t,i}$
with $\Pr \{ \kappa_{t,i} = 1 \} = \r$ for all $t,i$.
The ``rate'' $\r$ was assigned various values to measure performance.),
Option 1A and Option 4A
(Options 1 and 4 respectively, using an approximate gradient defined in
\eqref{eq:327} instead of the true gradient).
The noise was additive white Gaussian noise with various SNRs, which
determined the variance bound $M_t = M$ in \eqref{eq:328}.
The step and increment sequences were chosen to be
\bd
c_t = \frac{c_0}{(1+(t/\t))^q} , \alpha_t = \frac{\alpha_0}{(1+(t/\t))^p}  ,
\ed
where $\t = 200$, $c_0 = 0.01, \alpha_0 = 0.01$, $p = 1, q = 0.02$.
With these choices, the sufficient conditions of Theorems \ref{thm:2}
and \ref{thm:3} are satisfied. 
For comparision we set $\alpha_0$ to be same
in ADAM, NADAM, RMSPROP \cite{Tieleman-Hinton-coursera12}, and Heavy Ball (HB).
For NAG, we stick to step sequences proposed in \cite{Nesterov-Dokl83}.
The computations were conducted in \textit{python} using NumPy and PyTorch
(for GPU Acceleration) libraries.

Figure \ref{fig:3} displays the convergence of Option 4
(batch updating of a fraction $\r$ components, using noisy measurements
of the true gradient),
at 50dB SNR for various values of $\r$.
It can be seen that, with as few as 5\% of the components being updated
on average ($\r = 0.05$), the iterations converge.

Figure \ref{fig:1} shows the output of Option 1 (left side) and Option 4
with $\r = 0.2$ (right side).
Recall that Option 1 is to update \textit{all} components at each time step,
while Option 4 is to update a fraction $\r$ of the components,
using noise-corrupted true gradients in each case.
In this case, batch updating (BU) converges, albeit slowly.
Five other momentum-based (or second-order) methods fail to converge,
namely NAG which performs the worst, ADAM, NADAM and RMSPROP,
and HB which performs least badly.
Figure \ref{fig:2} shows the output of Option 1A (left side) and Option 4A
with $\r = 0.2$ (right side).
In this case, the noise-corrupted true gradients are replaced by
approximate gradients as in \eqref{eq:37}.
Here the Nesterov method actually diverges, while the others settle down
at some value of the objective function, though far from the optimum.
Batch updating converges to the optimum.

\begin{figure}[h]
\bc
\includegraphics[width=72.5mm]{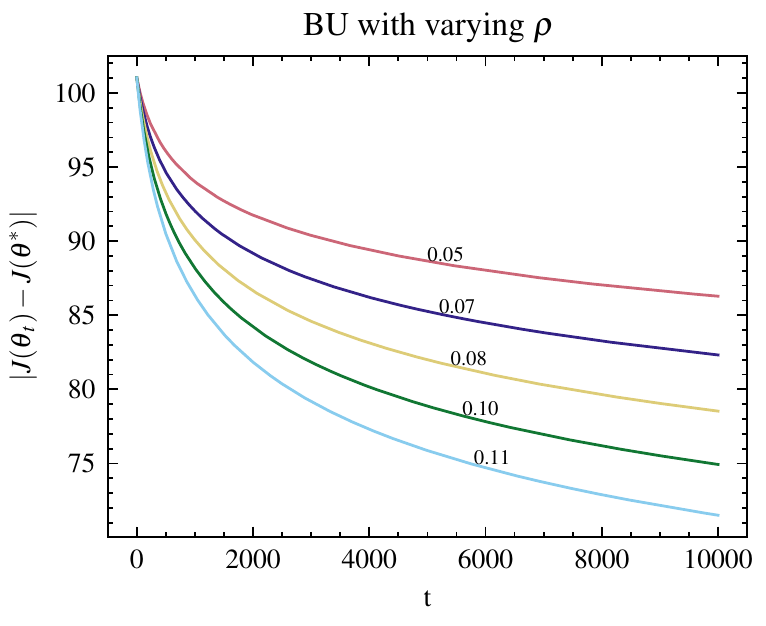}
\ec
\caption{Convergence of Batch Updating at Various Rates $\r$ and SNR = 50 dB.}
\label{fig:3}
\end{figure}

\section{Conclusions}\label{sec:Conc}

In this paper, we have presented a general and unified framework for
establishing the convergence of batch updating algorithms.
Then we have presented sufficient conditions for the convergence of
these batch updating algorithms.
Our method of proof is based on stochastic approximation theory,
specifically \cite{Robb-Sieg71}.
The results are applicable to nonconvex as well as convex objective
functions.
In particular, the search direction can be a batch-updated version of
any of the following:
\bit
\item A noisy measurement of the exact gradient. 
\item A first-order approximation of the gradient based on
noisy function measurements (also called zeroth-order or derivative-free
methods).
\eit

Numerical experiments show that when as little as 5\% of the coordinates
are updated at each iteration, and approximate gradients are used, 
batch updating
still converges, while batch momentum-based convex optimization methods
either diverge or fail to converge at quite low noise levels.
The existing convergence theory for
momentum-based algorithms does not apply to the
case of batch updating.
Extending the present theory to commonly used algorithms like ADAM/NADAM is a 
current topic of investigation.

\begin{figure*}[htbp]
\bc
\includegraphics[width=145mm]{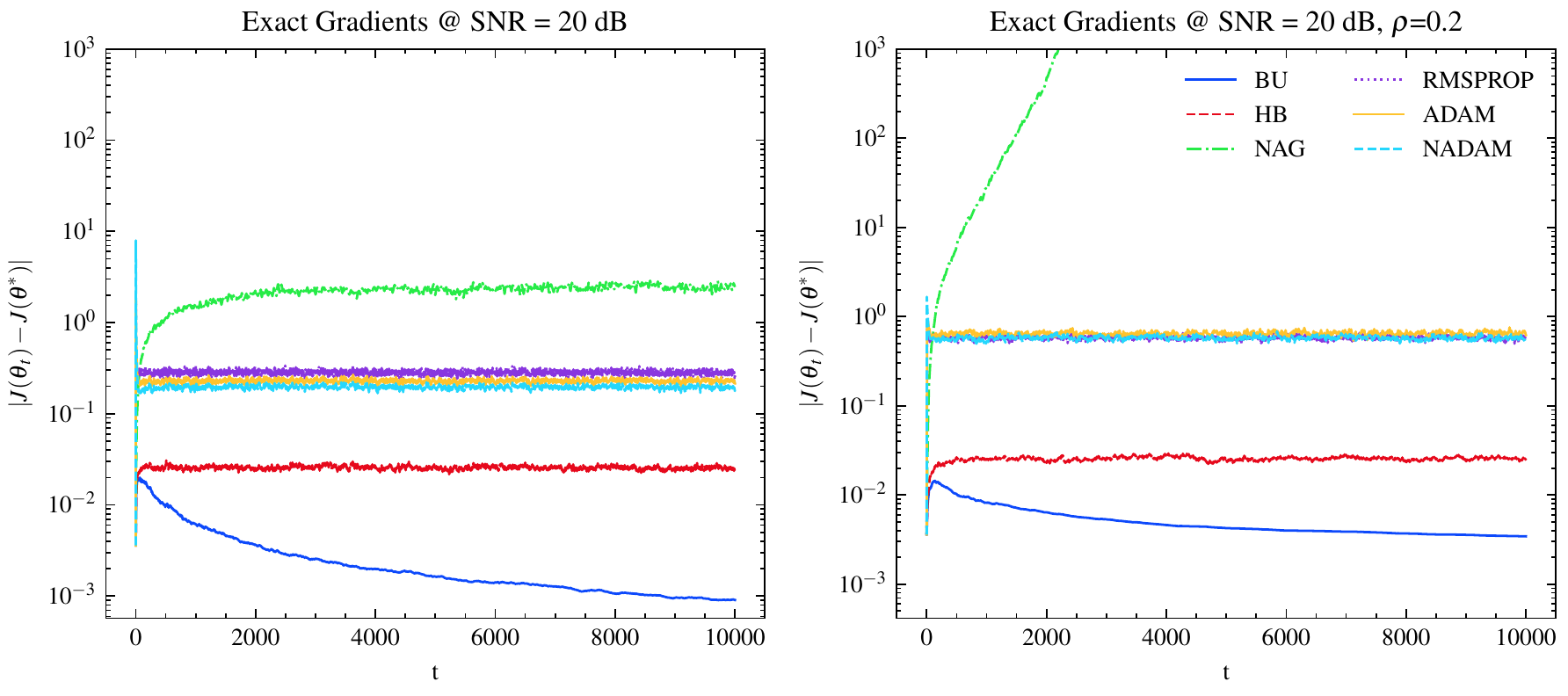}
\ec
\caption{[LEFT] Option 1; [RIGHT] Option 4.}
\label{fig:1}
\end{figure*}

\begin{figure*}[htbp]
\bc
\includegraphics[width=145mm]{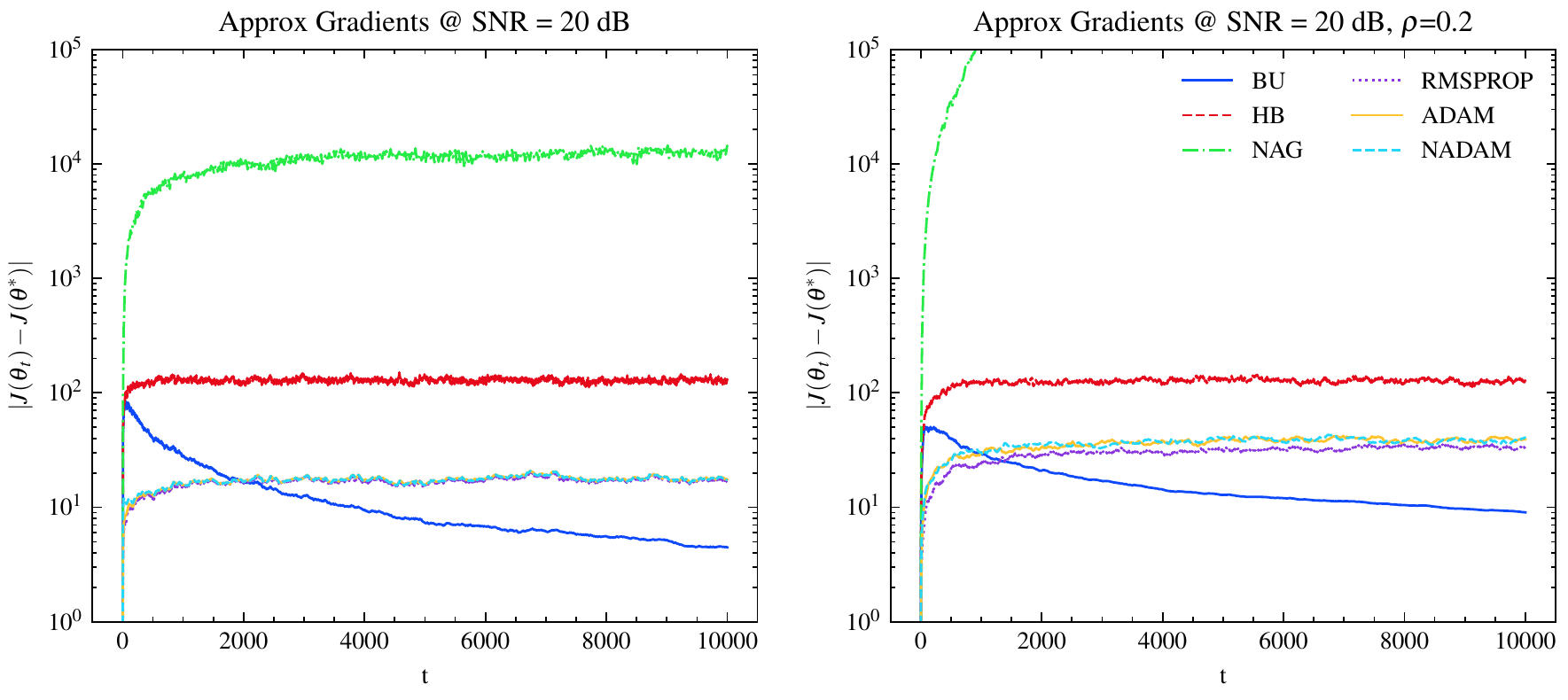}
\ec
\caption{[LEFT] Option 1A; [RIGHT] Option 4A}
\label{fig:2}
\end{figure*}


\bibliographystyle{ieeetr}
\bibliography{ML, Opt}

\end{document}

%% file: notation.tex
\def\eb{{\bf e}}
\def\f{{\bf f}}

\def\v{{\bf v}}

\def\x{{\bf x}}
\def\y{{\bf y}}
\def\z{{\bf z}}

\def\B{{\cal B}}
\def\C{{\cal C}}

\def\F{{\cal F}}

\def\M{{\cal M}}

\def\R{{\mathbb R}}

\def\al{\alpha}
\def\d{\delta}

\def\e{\epsilon}
\def\g{\gamma}

\def\l{\lambda}

\def\om{\omega}
\def\OM{\Omega}
\def\r{\rho}
\def\s{\sigma}
\def\SI{\Sigma}
\def\t{\tau}
\def\th{\theta}

\def\bth{{\boldsymbol \theta}}
\def\bphi{{\boldsymbol \phi}}
\def\bpsi{{\boldsymbol \psi}}

\def\bxi{{\boldsymbol \xi}}
\def\bzeta{{\boldsymbol \zeta}}

\def\cub{\bar{c}}
\def\clb{\underline{c}}

\def\Jb{\bar{J}}

\def\tai{t \ap \infty}

\def\ap{\rightarrow}

\def\bi{\{0,1\}}

\def\bz{{\bf 0}}
\def\imp{\; \Longrightarrow \;}

\def\fa{\; \forall}

\def\half{\frac{1}{2}}

\def\as{\mbox{ a.s.}}

\def\nm{\Vert}

\renewcommand{\and}{\mbox{$\wedge$}}

\newcommand{\bc}{\begin{center}}
\newcommand{\ec}{\end{center}}
\newcommand{\be}{\begin{equation}}
\newcommand{\ee}{\end{equation}}
\newcommand{\bd}{\begin{displaymath}}
\newcommand{\ed}{\end{displaymath}}
\newcommand{\ba}{\begin{array}}
\newcommand{\ea}{\end{array}}
\newcommand{\ben}{\begin{enumerate}}
\newcommand{\een}{\end{enumerate}}
\newcommand{\bit}{\begin{itemize}}
\newcommand{\eit}{\end{itemize}}
\newcommand{\beq}{\begin{eqnarray}}
\newcommand{\eeq}{\end{eqnarray}}
\newcommand{\btab}{\begin{tabular}}
\newcommand{\etab}{\end{tabular}}
\newcommand{\bfig}{\begin{figure}}
\newcommand{\efig}{\end{figure}}
\newcommand{\btp}{\begin{tikzpicture}}
\newcommand{\etp}{\end{tikzpicture}}

\newcommand{\nmeu}[1]{ \nm #1 \nm_2 }
\newcommand{\nmeusq}[1]{ \nm #1 \nm_2^2 }

\def\gJ{\nabla J}
\def\bths{\bth^*}
\def\bphit{\tilde{\bphi}}

\newcommand{\IP}[2]{ \langle #1 , #2 \rangle }

\def\nmsl1{\nm_{{\rm SL1}}}